\documentclass[11pt,a4paper,oneside]{article}
\usepackage[english]{babel}
\usepackage[T1]{fontenc}
\usepackage{amsmath}
\usepackage{amsthm}
\usepackage{amsfonts}
\usepackage{amssymb}
\usepackage{lmodern} %Type1-font for non-english texts and characters
\usepackage{indentfirst}		% Indenta o primeiro parÃ¡grafo de cada seÃ§Ã£o.
\usepackage{microtype} 			% para melhorias de justificaÃ§Ã£o
\usepackage[numbers]{natbib}
\usepackage{hyphenat}
\usepackage{pgfplots}  	
\usepackage[latin1]{inputenc}
\usepackage{indentfirst}

\definecolor{blackgreen}{RGB}{0,80,0}

\usepackage[plainpages=false,pdfpagelabels,pdfusetitle,pdfdisplaydoctitle, pdfstartpage=1, pdfstartview={{XYZ null null 1.00}}]{hyperref}
\hypersetup{
	pdftitle={Blow-up of non-radial solutions for the L2 critical inhomogeneous NLS equation},    
	pdfauthor={Mykael Cardoso and Luiz Gustavo Farah},    
	pdfnewwindow=true,      
	colorlinks=true,      
	linkcolor=blackgreen,         
	citecolor=red}

\usepackage[top=2cm, bottom=2cm, left=2cm, right=2cm]{geometry}
\usepackage{mathtools}
\mathtoolsset{showonlyrefs}

\newcommand{\Real}{\mathbb R}
\newtheorem{thm}{Theorem}[section]

\newtheorem{lemma}[thm]{Lemma}
\newtheorem{rem}[thm]{Remark}

\numberwithin{equation}{section}

\title{Blow-up of non-radial solutions for the $L^2$ critical inhomogeneous NLS equation}
\author{Mykael Cardoso and Luiz Gustavo Farah} % Autor
\date{} % Data

\begin{document}
\maketitle
	
\begin{abstract}\noindent
We consider the $L^2$ critical inhomogeneous nonlinear Schr\"odinger (INLS) equation in $\mathbb{R}^N$
\begin{align}\label{inls}
i \partial_t u +\Delta u +|x|^{-b} |u|^{\frac{4-2b}{N}}u = 0,
\end{align}
where $N\geq 1$ and $0<b<2$. We prove that if $u_0\in H^1(\mathbb{R}^N)$ satisfies $E[u_0]<0$, then the corresponding solution blows-up in finite time. This is in sharp contrast to the classical $L^2$ critical NLS equation where this type of result is only known in the radial case for $N\geq 2$. 
\end{abstract}

\section{Introduction}
In this work we consider the initial value problem (IVP) for the $L^2$ critical inhomogeneous nonlinear Schr\"odinger (INLS) equation
\begin{equation}
\begin{cases}
i \partial_t u + \Delta u + |x|^{-b} |u|^{\frac{4-2b}{N}}u = 0, \,\,\, x \in \mathbb{R}^N, \,t>0,\\
u(0) = u_0 \in H^1(\mathbb{R}^N),
\end{cases}
\label{PVI}
\end{equation}
where $N\geq 1$ and $0<b<2$. For $b=0$, \eqref{PVI} reduces to the IVP associated to the classical nonlinear Schr\"odinger (NLS) equation. This model is called $L^2$ critical since the scaling symmetry $u(x,t)\mapsto \lambda^{\frac{N}{2}}u(\lambda x, \lambda^2 t)$ leaves invariant the $L^2$ norm. The local well-posedness of the IVP \eqref{PVI} was obtained by \citet[Appendix K]{GENSTU} (see also \citet{CARLOS}). Moreover, \citet{G12}  also proved that this problem is globally well-posed below the ground state threshold.

The solutions of the IVP \eqref{PVI} satisfy mass and energy conservation laws given respectively by
\begin{equation}\label{Mass}
M[u(t)] =\int  |u(x,t)|^2\, dx =M[u_0]. 
\end{equation}
and
\begin{equation}\label{Energy}
E[u(t)] =\frac12 \int  |\nabla u(x,t)|^2\, dx - \frac{N}{4-2b+2N}\int |x|^{-b}|u(x,t)|^{\frac{4-2b}{N}+2}\,dx=E[u_0]. 
\end{equation}
Our main result is the following.
\begin{thm}\label{Blowup} Let $N\geq 1$ and $0<b<2$. If $u_0\in H^1(\mathbb{R}^N) $ and $E[u_0]< 0,$ then the corresponding solution $u(t)$ to \eqref{PVI} blows-up in finite.
\end{thm}

The previous theorem was first obtained by \citet{ogawa1991blow} for the classical radial NLS equation when $N\geq 2$ and, combining a scaling argument, the same authors in \cite{OT91PAMS} were able to improve this result for the non-radial NLS equation in dimension one. Applying the same ideas, \citet{dinh2017blowup} extended these results for the INLS model under identical conditions: radial for $N\geq 2$ and non-radial for $N=1$. Here we refine the argument of \citet{ogawa1991blow} to consider the non-radial INLS equation in all spatial dimensions, without relying on the scaling argument of \cite{OT91PAMS}. We should point out that Theorem \ref{Blowup} is still unknown for the classical NLS equation in the non-radial when $N\geq 2$.

%The previous theorem extends the results of \citet{dinh2017blowup} to the non-radial case. It is also unknown for the classical NLS equation. The ideas introduced by  \citet{ogawa1991blow} for the radial classical NLS equation works for the non-radial INLS equation and there is no need of the refined argument given by \citet{OT91PAMS} in the one dimensional case.

Recently, several papers reported results for the non-radial INLS equation that so far can only be obtained for the radial NLS equation \cite{BL21}, \cite{CC21}, \cite{CF21}, \cite{GM21} and \cite{M21}. The present paper is another contribution in this direction. The new tool in the INLS setting is the decaying factor $|x|^{-b}$ which implies a control, away from the origin, for the terms arising from the nonlinearity. For the NLS equation this type of control is usually made by an application of a radial Sobolev embedding due to \citet[Lemma 1]{Strauss}.  

This paper is organized as follows. In Section \ref{sec2}, we introduce the basic notation and established a non-radial interpolation estimate. The last  section is devoted to the proof of Theorem \ref{Blowup}.

\section{Notation and Preliminaries}\label{sec2}
In this section we introduce the basic notation used throughout the manuscript. The symbol $c$ will denote various positive constants and its exact value is not essential in our analysis. We write $a \lesssim b$ to denote $a \leq c\,b$ for some positive constant $c$. Similarly we define $a \gtrsim b$.  The spaces $L^p(\mathbb{R}^N)$ and $H^{1}(\mathbb{R}^N)$ will be abbreviated as $L^p$ and $H^{1}$ with the norms denoted by $\|\cdot\|_p$ and $\|\cdot\|_{H^{1}}:=\|\cdot\|_{2}+\|\nabla \cdot\|_{2}$, respectively. We also consider the functional space $W^{1,\infty}=\{f\in L^{\infty}, \nabla f\in L^{\infty}\}$.

%and $\|f\|_{\dot{H}^{s,p}}:=\|D^sf\|_{L^p},$
%where  $J^s$ and $D^s$ stand for the Bessel and Riesz potentials of order $s$, given via Fourier transform by $\widehat{J^s f}=(1+|\xi|^2)^{\frac{s}{2}}\widehat{f}$ and $\widehat{D^sf}=|\xi|^s\widehat{f}.
%$
%If $p=2$ we denote $H^{s,2}$ and $\dot{H}^{s,2}$ simply by $H^s$ and  $\dot{H}^{s}$, respectively.

Next, we obtain an interpolation estimate that will be very useful in the proof of our main result.

\begin{lemma}[Non-radial interpolation estimate] Let $N\geq 1$, $0<b <2$ and $\phi$ be a positive real valued function. 
\begin{itemize}
\item If $N\neq 2$ and $\phi^{\frac{1}{2-b}} \in W^{1,\infty}$, then for all $f\in H^1$ we have
\begin{equation}\label{Interp1}
\int \phi|f|^{\frac{4-2b}{N} +2}dx\lesssim \left(\| \nabla \left(\phi^{\frac{1}{2-b}}\right) f \|_2+\| \phi^{\frac{1}{2-b}} \nabla f\|_2\right)^{2-b}\|f\|_2^{\frac{4+b(N-2)}{N}}.
\end{equation}
\item If $N= 2$ and $\phi^{\frac{1}{2-\frac{b}{2}}} \in W^{1,\infty}$, then for all $f\in H^1$ we have
\begin{equation}\label{Interp2}
\int \phi|f|^{4-b}dx\lesssim \left(\| \phi^{\frac{1}{2-\frac{b}{2}}}  f\|_2+\| \nabla (\phi^{\frac{1}{2-\frac{b}{2}}}) f \|_2+\| \phi^{\frac{1}{2-\frac{b}{2}}} \nabla f\|_2\right)^{2-\frac{b}{2}}\|f\|_2^{2-\frac{b}{2}}.
\end{equation}
\end{itemize}
\end{lemma}

\begin{proof}
For $N\geq 3$, we apply the Holder inequality and Sobolev embedding to obtain
\begin{eqnarray*}
\int \phi|f|^{\frac{4-2b}{N} +2}dx&=&\int (\phi|f|^{2-b})(|f|^{\frac{4-2b}{N} +b})dx\\
&\lesssim & \|\phi |f|^{2-b}\|_{\frac{2N}{(N-2)(2-b)}}\||f|^{\frac{4-2b}{N} +b}\|_{\frac{2N}{4+b(N-2)}}\\
&=& \|\phi^{\frac{1}{2-b}}f\|^{2-b}_{\frac{2N}{N-2}}\|f\|^{\frac{4+b(N-2)}{N} }_{2}\\
&\lesssim & \|\nabla (\phi^{\frac{1}{2-b}}f)\|^{2-b}_{2}\|f\|^{\frac{4+b(N-2)}{N} }_{2},
\end{eqnarray*}
which implies the desired inequality.

Next when $N=1$, we first claim that
\begin{equation}\label{OTN1}
\|\phi^{\frac{1}{4-2b}}f\|_{\infty}\lesssim \|f\|^{1/2}_2\left(\|\left(\phi^{\frac{1}{2-b}}\right)' f\|_2+\|\phi^{\frac{1}{2-b}} f'\|_2\right)^{1/2}.
\end{equation}
Indeed, by an approximation argument we may assume that $f$ has compact support and therefore
\begin{eqnarray*}
\phi^{\frac{1}{2-b}}f^2(x)&=&\frac{1}{2}\left(\int_{-\infty}^{x}\left(\phi^{\frac{1}{2-b}}f^2\right)'ds+\int_{x}^{+\infty}\left(\phi^{\frac{1}{2-b}}f^2\right)'ds \right)\\
&\lesssim& \int \left(\phi^{\frac{1}{2-b}}\right)'f^2ds +\int \phi^{\frac{1}{2-b}} f' fds \\
&\lesssim & \|f\|_2\left(\|\left(\phi^{\frac{1}{2-b}}\right)'f\|_2+\|\phi^{\frac{1}{2-b}} f'\|_2\right)
\end{eqnarray*}
and \eqref{OTN1} is proved. Using this inequality we have
\begin{eqnarray*}
\int \phi|f|^{4-2b +2}dx&=&\int \left|\phi^{\frac{1}{4-2b}}f\right|^{4-2b}|f|^{2}dx\\
&\lesssim & \|f\|^{2-b}_2\left(\| \left(\phi^{\frac{1}{2-b}}\right)' f\|_2+\| \phi^{\frac{1}{2-b}} f'\|_2\right)^{2-b} \|f\|^2_{2},
\end{eqnarray*}
so \eqref{Interp1} also holds in this case.

Finally, we consider the case $N=2$ and use the following Sobolev embedding (see, for instance, \citet[Proposition 4.18]{DD12})
\begin{equation}\label{SEI1} % SEI=sobol embedd ineq.
\|f\|_{L^r}\leq c\|f\|_{H^{1}}, \,\, \mbox{for all}\,\, r\in[2,+\infty).
\end{equation}
together with the Holder inequality to obtain
\begin{eqnarray*}
\int \phi|f|^{4-b}dx&=&\int |\phi^{\frac{1}{2-\frac{b}{2}}}f|^{2-\frac{b}{2}}|f|^{2-\frac{b}{2}}dx\\
&\lesssim & \||\phi^{\frac{1}{2-\frac{b}{2}}}f|^{2-\frac{b}{2}}\|_{\frac{4}{b}} \||f|^{2-\frac{b}{2}}\|_{\frac{2}{2-\frac{b}{2}}}\\
&= &\|\phi^{\frac{1}{2-\frac{b}{2}}}f\|_{\frac{8-2b}{b}}^{2-\frac{b}{2}}\|f\|_{2}^{2-\frac{b}{2}}\\
&\lesssim & \|\phi^{\frac{1}{2-\frac{b}{2}}}f\|_{H^1}^{2-\frac{b}{2}}\|f\|_{2}^{2-\frac{b}{2}}.
\end{eqnarray*}
Thus, from the definition of the $H^1$-norm we deduce inequality \eqref{Interp2}.
\end{proof}

\begin{rem}
The classical Gagliardo-Nirenberg (see for instance \citet{W83}, inequality (I.2))
$$
\int | f|^{2\sigma+2}dx\leq C \|\nabla f\|^{N\sigma}_2\|f\|_2^{2+\sigma(2-N)}, \quad \mbox{if} \quad 0<\sigma<\frac{2}{N-2}
$$
implies, for $\sigma = \frac{2-b}{N}$ and  assuming $\phi^{\frac{1}{\frac{4-2b}{N}+2}} f\in H^1$, that
\begin{equation}\label{Interp3}
\int \phi | f|^{\frac{4-2b}{N}+2}dx\leq C \|\nabla (\phi^{\frac{1}{\frac{4-2b}{N}+2}} f)\|^{2-b}_2\|\phi^{\frac{1}{\frac{4-2b}{N}+2}} f\|_2^{\frac{4+b(N-2)}{N}}.
\end{equation}
The main difference between inequalities \eqref{Interp1}-\eqref{Interp2} and \eqref{Interp3} is the power of the function $\phi$. As we will see later, to prove our main result we need this power to be greater then $1/2$ and therefore inequality \eqref{Interp3} is not enough to close the argument.
\end{rem}

%We also recall the Sobolev inequalities.
%\begin{lemma}\textbf{(Sobolev embedding)}\label{SI} Let $s\in (0,+\infty)$ and $1\leq p<+\infty$.
%\begin{itemize}
%\item [(i)] If $s\in (0,\frac{N}{p})$ then $H^{s,p}(\mathbb{R}^N)$ is continuously embedded in $L^r(\mathbb{R^N})$ where $s=\frac{N}{p}-\frac{N}{r}$. Moreover, 
%\begin{equation}\label{SEI} % SEI=sobol embedd ineq.
%\|f\|_{L^r}\leq C(N,s)\|D^sf\|_{L^{p}}.
%\end{equation}
%\item [(ii)] If $s=\frac{N}{2}$ then $H^{s}(\mathbb{R}^N)\subset L^r(\mathbb{R^N})$ for all $r\in[2,+\infty)$. Furthermore,
%\begin{equation}\label{SEI1} % SEI=sobol embedd ineq.
%\|f\|_{L^r}\leq c\|f\|_{H^{s}}.
%\end{equation}
%\end{itemize}
%\begin{proof}  \citet[Proposition 4.18]{DD12}). 
%\end{proof}
%\end{lemma}
%In particular we have 
%\begin{align}\label{SEsc}
%\|f\|_{L^p}\leq C(N,s)\|f\|_{\dot H^s},\quad\quad\forall f\in \dot H^s(\Real^N),
%\end{align}
%where $p=\frac{2N}{N-2s}$. Note that $\dot H^{s_c}\subset L^{\sigma_c}$ since $\sigma_c=\frac{2N\sigma}{2-b}=\frac{2N}{N-2s_c}.$

\section{The proof of Theorem \ref{Blowup}}\label{sec3}
Let $u_0\in H^1$ such that $E[u_0]<0$ and assume by contradiction that the corresponding solution $u(t)$ of \eqref{PVI} exists globally in time. For a bounded non-negative radial function $\phi\in C^{\infty}(\Real^N)$, define
%, such that
%\begin{align}\label{phi}
%	\phi(x)=
%	\left\{
%	\begin{array}{ll}
%		|x|^2,&\mbox{ se }|x|\leq 2\\
%		constant,&\mbox{ se }|x|\geq 4
%	\end{array}
%	\right.
%\end{align}
%satisfying
%\begin{align}\label{nablaphi}
%	\phi(x)\leq c|x|^2, \quad |\nabla \phi (x)|^2\leq c\phi(x) \quad \mbox{and} \quad \partial_r^2\phi(x)\leq 2, \quad \mbox{for all } x\in \Real^N,
%\end{align}
%with $r=|x|$. 
 $\phi_R(x)=R^2\phi\left(\frac{x}{R}\right)$ and
\begin{align}\label{virial}
	z_R(t)=\displaystyle\int\phi_R|u(t)|^2\,dx,
\end{align} 
for $R>0$ to be chosen later. It is clear that
$$
z_R(t)\leq R^2\|\phi\|_{\infty}\|u_0\|_2^2,
$$
by the mass conservation \eqref{Mass}.

From direct computations (see, for instance, Proposition 7.2 in \cite{FG20}), we have the following virial identities
\begin{equation}\label{zR'2}
	z_R'(t)=2\mbox{Im}\int \nabla\phi_R\cdot\nabla u(t)\overline{u}(t)\,dx
\end{equation}
and 
\begin{align}\label{zR''22}
	z_R''(t)=&4\mbox{Re} \sum_{j,k=1}^{N}\int \partial_ju(t)\,\partial_k\overline u(t)\,\partial^2_{jk}\phi_R\,dx-\int |u(t)|^2 \Delta^2\phi_R\nonumber\\
	&-\frac{4-2b}{N+2-b}\int|x|^{-b}|u(t)|^{\frac{4-2b}{N}+2}\Delta\phi_R\,dx\nonumber\\
	&+\frac{2N}{N+2-b}\int\nabla\left(|x|^{-b}\right)\cdot \nabla\phi_R|u(t)|^{\frac{4-2b}{N}+2}\,dx.
\end{align}

Recall that
$$
\partial_j=\frac{x_j}{r}\partial_r \,\,\, \mbox{and}\,\,\, \partial_{kj}=\left(\frac{\delta_{kj}}{r}-\frac{x_kx_j}{r^3}\right)\partial_r + \frac{x_kx_j}{r^2}\partial^2_r
$$
where $\partial_r$ denotes the radial derivative with respect to $r=|x|$. From these relations, since $\phi$ is radial, we deduce
\begin{eqnarray*}
\sum_{j,k=1}^{N}\partial_ju\,\partial_k\overline u\,\partial^2_{jk}\phi_R&=& \sum_{j,k=1}^{N}\partial_ju\,\partial_k\overline u\left[\left(\frac{\delta_{kj}}{r}-\frac{x_kx_j}{r^3}\right)\partial_r\phi + \frac{x_kx_j}{r^2}\partial^2_r\phi\right]\\
&=& \frac{\partial_r\psi}{r}|\nabla u|^2+\left(\frac{\partial_r^2\phi}{r^2}-\frac{\partial_r\phi}{r^3}\right)\left(\sum_{j,k=1}^{N} (x_j\partial_ju)\,\overline{(x_k\partial_ku)}\right)\\
&=& \frac{\partial_r\psi}{r}|\nabla u|^2+\left(\frac{\partial_r^2\phi}{r^2}-\frac{\partial_r\phi}{r^3}\right)|x\cdot \nabla u|^2.
\end{eqnarray*}
Moreover, it is easy to see that
$$
\Delta \phi = \frac{N-1}{r}\partial_r\phi +\partial^2_r\phi
$$
and
$$
\nabla\left(|x|^{-b}\right) \cdot \nabla \phi=-b|x|^{-b}\frac{\partial_r\psi}{r},
$$
since $\nabla\left(|x|^{-b}\right)=-b|x|^{-b-2}x$.

Therefore, we can rewrite the identities \eqref{zR'2}-\eqref{zR''22} as
\begin{align}\label{virial1}
	z'(t)=2\,\mbox{Im}\int \partial_r\phi_R\frac{x\cdot \nabla u(t)}{r}\overline{u}(t)\,dx
\end{align}
and 
\begin{align}\label{virial2}
	z_R ''(t)=&\,\,4\int \frac{\partial_r\phi_R}{r}|\nabla u(t)|^2\,dx+4\int \left(\frac{\partial_r^2\phi_R}{r^{2}}-\frac{\partial_r \phi_R}{r^3}\right)|x\cdot \nabla u(t)|^2\,dx-\int|u(t)|^2 \Delta^2\phi_R
	\,dx \nonumber\\
	&+\frac{4-2b}{N+2-b}\int \left[-\partial^2_r\phi_R -\left(N-1+\frac{bN}{2-b}\right)\frac{\partial_r \phi_R}{r}\right]|x|^{-b}|u(t)|^{\frac{4-2b}{N}+2}\,dx.
\end{align}

Continuing from above, we use the energy conservation \eqref{Energy} to obtain
\begin{align}\label{zR''}
	z_R''(t)=2E[u_0]+K_1+K_2+K_3,
\end{align}
where
\begin{align}
	K_1=&-4\int \left(2-\frac{\partial_r\phi_R}{r}\right)|\nabla u(t)|^2\,dx-4\int \left(\frac{\partial_r \phi_R}{r^3}-\frac{\partial^2_r\phi_R}{r^2}\right)|x\cdot \nabla u(t)|^2\,dx,\\
	K_2=&\frac{2}{N+2-b}\int\left[(2-b)(2-\partial^2_r\phi_R)+(2N-2+b)\left(2-\frac{\partial_r \phi_R}{r}\right)\right]|x|^{-b}|u(t)|^{\frac{4-2b}{N}+2}\,dx\\
	K_3=&-\int|u(t)|^2\Delta^2\phi_R\,dx.
\end{align}

Now, we define a function $\phi_R$ such that
\begin{equation}\label{phicond}
\partial_r\phi_R(r)-r{\partial^2_r\phi_R(r)}\geq 0, \,\, \mbox{for all} \,\, r=|x|\in \mathbb{R}.
\end{equation}
Indeed, inspired by the work of \citet{ogawa1991blow}, we first consider, for $k\in \mathbb{N}$ to be chosen later, the following function
\begin{align}\label{v(r)}
	v(r)=
	\left\{
	\begin{array}{ll}
		2r, &\mbox{ if } 0\leq r \leq 1\\
		2r-2(r-1)^k, &\mbox{ if } 1< r \leq 1+\left(\frac{1}{k}\right)^{\frac{1}{k-1}}\\
		\mbox{smooth and}\,\, v'<0, &\mbox{ if } 1+\left(\frac{1}{k}\right)^{\frac{1}{k-1}}< r < 2\\
		0,&\mbox{ if }r\geq 2.
	\end{array}
	\right.
\end{align}
\begin{rem}
Note that the function $f(r)=2r-2(r-1)^k$ for $r\geq 1$ has an absolute maximum at $r=1+\left(\frac{1}{k}\right)^{\frac{1}{k-1}}$ and therefore $f'\left(1+\left(\frac{1}{k}\right)^{\frac{1}{k-1}}\right)=0$.
\end{rem}
Define the radial function
$$
\phi(r)=\int_0^{r}v(s)ds 
$$
Recall that $\phi_R(r)=R^2\phi\left(\frac{r}{R}\right)$, which implies
\begin{equation}\label{phiR}
\partial_r\phi_R(r)=Rv\left(\frac{r}{R}\right) \quad \mbox{and} \quad \partial^2_r\phi_R(r)=v'\left(\frac{r}{R}\right).
\end{equation}
It is easy to see that inequality \eqref{phicond} holds for $0<r\leq R$ and $r\geq 2R$ by direct computation and for $R\left(1+\left(\frac{1}{k}\right)^{\frac{1}{k-1}}\right)< r < 2R$ by \eqref{phiR} and the fact that $v'<0$ and $v\geq 0$ in this region. It remains to consider the region $R<r\leq R\left(1+\left(\frac{1}{k}\right)^{\frac{1}{k-1}}\right)$. By the definition of $v$ \eqref{v(r)} and relations \eqref{phiR}, in this region we have
\begin{equation}\label{phiR-int}
\partial_r\phi_R(r)=2R\left[\frac{r}{R}-\left(\frac{r}{R}-1\right)^k\right] \quad \mbox{and} \quad \partial^2_r\phi_R(r)=2-2k\left(\frac{r}{R}-1\right)^{k-1}.
\end{equation}
Thus
\begin{eqnarray*}
\partial_r\phi_R(r)-{r\partial^2_r\phi_R(r)}&=&2r\left(\frac{r}{R}-1\right)^{k-1}\left[k-\frac{R}{r}\left(\frac{r}{R}-1\right)\right]\\
&\geq &2r\left(\frac{r}{R}-1\right)^{k-1}\left[k-\left(\frac{1}{k}\right)^{\frac{1}{k-1}}\right]>0.
\end{eqnarray*}

In view of inequality \eqref{phicond}, the second integral in the definition of $K_1$ is positive and since $\partial_r\phi_R(r)=2r$ for $0<r\leq R$ we obtain
\begin{align}\label{K_1}
	K_1\leq &-\int_{|x|> R}\Phi_{1,R}|\nabla u(t)|^2\,dx,
\end{align}
where $\Phi_{1,R}=4\left(2-\frac{\partial_r\phi_R}{r}\right)$.

Moreover, using the $\partial^2_r\phi_R(r)=2$ for $0<r\leq R$, we have
\begin{align}
	\mbox{supp} \left[(2-b)(2-\partial^2_r\phi_R)+(2N-2+b)\left(2-\frac{\partial_r \phi_R}{r}\right)\right]\subset (R,\infty),
\end{align}
which implies
$$
K_2=\int_{|x|> R}\Phi_{2,R}|x|^{-b}|u(t)|^{\frac{4-2b}{N}+2}\,dx,
$$
where $\Phi_{2,R}=\frac{2}{N+2-b}\left[(2-b)(2-\partial^2_r\phi_R)+(2N-2+b)\left(2-\frac{\partial_r \phi_R}{r}\right)\right]$. 

It is clear that $\Phi_{1,R}(r), \Phi_{2,R}(r)\geq 0$, since by definition $\partial^2_r\phi_R(r)\leq 2$ and $\partial_r\phi_R(r)\leq 2r$ for all $r=|x|\in \mathbb{R}$.

Now we use the decay of $|x|^{-b}$ away from the origin to estimate $K_2$. To fix the ideas we only consider the case $N\neq 2$. When $N=2$ the proof is completely analogous just applying inequality \eqref{Interp2} instead of \eqref{Interp1} (see also Remark \ref{N=2} below for more details). From the inequality \eqref{Interp1} and Young's inequality, we deduce
\begin{eqnarray*}
K_2&\lesssim & \frac{1}{R^b}\int_{|x|> R}\Phi_{2,R}|u(t)|^{\frac{4-2b}{N}+2}\,dx\\
&\lesssim &  \frac{1}{R^b} \left(\|\nabla (\Phi_{2,R}^{\frac{1}{2-b}})u(t)\|_2+\|\Phi_{2,R}^{\frac{1}{2-b}}\nabla u(t)\|_2 \right)^{2-b}\|u_0\|_2^{\frac{4-2b+Nb}{N}}\\
&\lesssim &  \varepsilon \left(\|\nabla (\Phi_{2,R}^{\frac{1}{2-b}})u(t)\|_2+\|\Phi_{2,R}^{\frac{1}{2-b}}\nabla u(t)\|_2 \right)^{2}   +\frac{\|u_0\|_2^{\frac{2(4-2b+Nb)}{bN}}}{\varepsilon^{\frac{2-b}{b}}R^2}.
\end{eqnarray*}
It is clear that $\Phi_{2,R}\in L^{\infty}$. We claim that 
\begin{equation}\label{nablaPhi}
\left|\nabla \left(\Phi_{2,R}^{\frac{1}{2-b}}(r)\right)\right|\lesssim \frac{1}{R},\,\,\mbox{for all}\,\, r=|x|\in \mathbb{R}.
\end{equation}
Indeed if $r\leq R$ and $r\geq 2R$, then $\nabla \left(\phi^{\frac{1}{2-b}}(r)\right)=0$ and the desired inequality holds. In the intermediary region we first consider $R<r\leq R\left(1+\left(\frac{1}{k}\right)^{\frac{1}{k-1}}\right)$, where, in view of \eqref{phiR-int}, we obtain
\begin{eqnarray*}
\left|\nabla \left(\Phi_{2,R}^{\frac{1}{2-b}}(r)\right)\right|&=&\left|\partial_r\left[\left(\frac{r}{R}-1\right)^{\frac{k-1}{2-b}}\left(\frac{4}{N+2-b}\left[k(2-b)+(2N-2+b)\left(1-\frac{R}{r}\right)\right]\right)^{\frac{1}{2-b}}\right]\right|\\
&\lesssim & \frac{1}{R}\left(\frac{r}{R}-1\right)^{\frac{k-1}{2-b}-1}+\frac{R}{r^2}\left(\frac{r}{R}-1\right)^{\frac{k-1}{2-b}} \\
&\lesssim & \frac{1}{R},
\end{eqnarray*}
if we assume that $\frac{k-1}{2-b}-1>0$ or $k>3-b$. Finally, when $ R\left(1+\left(\frac{1}{k}\right)^{\frac{1}{k-1}}\right)<r< 2R$ we have that $\Phi_{2,R}(r) \gtrsim 1$ and from \eqref{phiR} we deduce
\begin{eqnarray*}
\left|\nabla \left(\Phi_{2,R}^{\frac{1}{2-b}}(r)\right)\right|&=&\left|\partial_r\left[\left(\frac{2}{N+2-b}\left[(2-b)\left(2-v'\left(\frac{r}{R}\right)\right)+(2N-2+b)\left(2-\frac{R}{r}v\left(\frac{r}{R}\right)\right)\right]\right)^{\frac{1}{2-b}}\right]\right|\\
&\lesssim &  \Phi_{2,R}^{\frac{b-1}{2-b}}(r)\left(\frac{1}{R}\left|v''\left(\frac{r}{R}\right)\right|+\frac{R}{r^2}\left|v\left(\frac{r}{R}\right)\right|+\frac{1}{r}\left|v'\left(\frac{r}{R}\right)\right|\right)\\
&\lesssim & \frac{1}{R}.
\end{eqnarray*}
Returning to the bound of $K_2$ and using estimate \eqref{nablaPhi} we get
\begin{align}\label{K_2}
K_2\lesssim & \,\, \varepsilon \left(\|\nabla (\Phi_{2,R}^{\frac{1}{2-b}})u(t)\|^{2}_2+\|\Phi_{2,R}^{\frac{1}{2-b}}\nabla u(t)\|^{2}_2 \right)  +\frac{\|u_0\|_2^{\frac{2(4-2b+Nb)}{bN}}}{\varepsilon^{\frac{2-b}{b}}R^2}\nonumber  \\ 
\lesssim &\,\, \varepsilon \int_{|x|> R}\Phi_{2,R}^{\frac{2}{2-b}}\left|\nabla u(t)\right|^2dx +\frac{\varepsilon \|u_0\|^2}{R^2}+\frac{\|u_0\|_2^{\frac{2(4-2b+Nb)}{bN}}}{\varepsilon^{\frac{2-b}{b}}R^2}.
\end{align}

Finally, since $\|\Delta^2\phi_R\|_{\infty}\lesssim 1/R^2$, we thus obtain from the mass conservation \eqref{Mass} the crude estimate
\begin{equation}\label{K_3}
K_3\lesssim \frac{\|u_0\|_2^2}{R^2}.
\end{equation}

Inserting estimates \eqref{K_1}, \eqref{K_2} and \eqref{K_3} into the right hand side of \eqref{zR''}, we infer that there exists $c>0$ such that
\begin{align}\label{zR''2}
z_R''(t)\leq& \,\,2E[u_0]+\int_{|x|> R} \left(c\varepsilon\Phi_{2,R}^{\frac{2}{2-b}}-\Phi_{1,R}\right)\left|\nabla u(t)\right|^2dx \nonumber \\
&+c\frac{\varepsilon\|u_0\|_2^2}{R^2}+c\frac{\|u_0\|_2^{\frac{2(4-2b+Nb)}{bN}}}{\varepsilon^{\frac{2-b}{b}}R^2}.
\end{align}

Next we claim that for sufficiently small $\varepsilon>0$ 
\begin{equation}\label{Phivare}
c\varepsilon\Phi_{2,R}^{\frac{2}{2-b}}(r)-\Phi_{1,R}(r)\leq 0, \,\,\mbox{for all}\,\, r> R.
\end{equation}

We first consider the region $R<r\leq R\left(1+\left(\frac{1}{k}\right)^{\frac{1}{k-1}}\right)$. By relations \eqref{phiR-int}, in this region we have
\begin{align}
c\varepsilon\Phi_{2,R}^{\frac{2}{2-b}}(r)-\Phi_{1,R}(r)=&c\varepsilon\left(\frac{r}{R}-1\right)^{\frac{2(k-1)}{2-b}}\left(\frac{4}{N+2-b}\left[k(2-b)+(2N-2+b)\left(1-\frac{R}{r}\right)\right]\right)^{\frac{2}{2-b}}\\
&-8\frac{R}{r}\left(\frac{r}{R}-1\right)^{k} \\
= &\left(\frac{r}{R}-1\right)^{k}\left[c\varepsilon\left(\frac{r}{R}-1\right)^{\frac{bk-2}{2-b}}\left(\frac{4}{N+2-b}\left[k(2-b)\right.\right.\right.\\
&+\left.\left.\left.(2N-2+b)\left(1-\frac{R}{r}\right)\right]\right)^{\frac{2}{2-b}}-8\frac{R}{r}\right]\\
\leq & \left(\frac{r}{R}-1\right)^{k}\left[c\varepsilon\left(\frac{1}{k}\right)^{\frac{bk-2}{(2-b)(k-1)}}\left(\frac{4}{N+2-b}\left[k(2-b)\right.\right.\right.\\
&+\left.\left.\left.(2N-2+b)\left(\frac{1}{k}\right)^{\frac{1}{k-1}}\right]\right)^{\frac{2}{2-b}}-\frac{8}{1+\left(\frac{1}{k}\right)^{\frac{1}{k-1}}}\right],
\end{align}
if we assume that $bk-2>0$ or $k>2/b$. So, we can chose $\varepsilon>0$ sufficiently small such that \eqref{Phivare} holds in this case.

Now we turn our attention to the region $R\left(1+\left(\frac{1}{k}\right)^{\frac{1}{k-1}}\right)< r$. Since $v'(r/R)\leq 0$ in this region from \eqref{phiR} we first have
\begin{equation}\label{Reg2}
\frac{\partial_r \phi_R}{r}(r)= \frac{R}{r}v\left(\frac{r}{R}\right)\leq \frac{1}{1+\left(\frac{1}{k}\right)^{\frac{1}{k-1}}}v\left(1+\left(\frac{1}{k}\right)^{\frac{1}{k-1}}\right)= 2-2\frac{\left(\frac{1}{k}\right)^{\frac{k}{k-1}}}{1+\left(\frac{1}{k}\right)^{\frac{1}{k-1}}}
\end{equation}
and
$$
\left|\partial^2_r\phi_R(r)\right|=\left|v'\left(\frac{r}{R}\right)\right|\leq \|v'\|_{\infty}.
$$
The last two inequalities imply $\Phi_{2,R}(r)\lesssim 1$. Moreover, from \eqref{Reg2}, we have
$$
\Phi_{1,R}(r)=4\left(2-\frac{R}{r}v\left(\frac{r}{R}\right)\right)\geq 8\frac{\left(\frac{1}{k}\right)^{\frac{k}{k-1}}}{1+\left(\frac{1}{k}\right)^{\frac{1}{k-1}}}.
$$
Therefore there exists $\varepsilon>0$ such that \eqref{Phivare} also holds in this case.

Note that in both cases considered above $\varepsilon>0$ was chosen independent of $R>0$. Finally, collecting estimates \eqref{zR''2}, \eqref{Phivare} and taking $R>0$ sufficiently large we deduce
$$
\frac{d^2}{dt^2}\displaystyle\int\phi_R|u(t)|^2\,dx= z_R''(t)\leq E[u_0]<0,
$$
and standard arguments imply that the solution blows-up in finite time concluding the proof of Theorem \ref{Blowup}.
\begin{rem}\label{N=2}
In the case $N=2$, the inequality \eqref{zR''2} may be replaced by
\begin{align}\label{zR''3}
z_R''(t)\leq& \,\,2E[u_0]+\int_{|x|> R} \left(c\varepsilon\Phi_{2,R}^{\frac{2}{2-\frac{b}{2}}}-\Phi_{1,R}\right)\left|\nabla u(t)\right|^2dx \nonumber \\
&+c\frac{\varepsilon\|u_0\|_2^2}{R^2}+c\frac{\|u_0\|_2^{\frac{2(4-b)}{b}}}{\varepsilon^{\frac{4-b}{b}}R^4}+c\frac{\|u_0\|_2^{4-b}}{R^b}.
\end{align}
taking into account the inequality \eqref{Interp2} to estimate $K_2$. The same arguments employed above implies 
\begin{equation*}\label{Phivare2}
c\varepsilon\Phi_{2,R}^{\frac{2}{2-\frac{b}{2}}}(r)-\Phi_{1,R}(r)\leq 0, \,\,\mbox{for all}\,\, r> R,
\end{equation*}
for sufficiently small $\varepsilon>0$, independent of $R>0$, as long as $k> 4/b$ for $v(r)$ given by \eqref{v(r)}. Moreover, the last three terms in the right hand side of \eqref{zR''3} can be made small, for $R>0$ sufficiently large, concluding the proof Theorem \ref{Blowup} also in this case.
\end{rem}

\vspace{0.5cm}
\noindent 
\textbf{Acknowledgments.} M.C. was partially supported by Coordena\c{c}\~ao de Aperfei\c{c}oamento de Pessoal de N\'ivel Superior - CAPES. L.G.F. was partially supported by Coordena\c{c}\~ao de Aperfei\c{c}oamento de Pessoal de N\'ivel Superior - CAPES, Conselho Nacional de Desenvolvimento Cient\'ifico e Tecnol\'ogico - CNPq and Funda\c{c}\~ao de Amparo a Pesquisa do Estado de Minas Gerais - Fapemig/Brazil.

%\bibliography{biblio2}

\begin{thebibliography}{15}
\providecommand{\natexlab}[1]{#1}
\providecommand{\url}[1]{\texttt{#1}}
\expandafter\ifx\csname urlstyle\endcsname\relax
  \providecommand{\doi}[1]{doi: #1}\else
  \providecommand{\doi}{doi: \begingroup \urlstyle{rm}\Url}\fi

\bibitem[Bai and Li(2021)]{BL21}
R.~Bai and B.~Li.
\newblock Blow-up for the inhomogeneous nonlinear {S}chr\"{o}dinger equation.
\newblock \emph{ArXiv preprint arXiv:2103.13214}, 2021.

\bibitem[Campos and Cardoso(2021)]{CC21}
L.~Campos and M.~Cardoso.
\newblock A virial-morawetz approach to scattering for the non-radial
  inhomogeneous {NLS}.
\newblock \emph{ArXiv preprint arXiv:2104.11266 (to appear in Proc. Amer. Math.
  Soc.)}, 2021.

\bibitem[Cardoso and Farah(2021)]{CF21}
M.~Cardoso and L.~G. Farah.
\newblock Blow-up solutions of the intercritical inhomogeneous {NLS} equation:
  the non-radial case.
\newblock \emph{ArXiv preprint arXiv:2105.10748}, 2021.

\bibitem[Demengel and Demengel(2012)]{DD12}
F.~Demengel and G.~Demengel.
\newblock \emph{Functional spaces for the theory of elliptic partial
  differential equations}.
\newblock Universitext. Springer, London; EDP Sciences, Les Ulis, 2012.
\newblock ISBN 978-1-4471-2806-9; 978-2-7598-0698-0.
\newblock \doi{10.1007/978-1-4471-2807-6}.
\newblock URL \url{http://dx.doi.org/10.1007/978-1-4471-2807-6}.

\bibitem[Dinh(2018)]{dinh2017blowup}
V.~D. Dinh.
\newblock Blowup of {$H^1$} solutions for a class of the focusing inhomogeneous
  nonlinear {S}chr\"{o}dinger equation.
\newblock \emph{Nonlinear Anal.}, 174:\penalty0 169--188, 2018.
\newblock ISSN 0362-546X.
\newblock \doi{10.1016/j.na.2018.04.024}.
\newblock URL \url{https://doi.org/10.1016/j.na.2018.04.024}.

\bibitem[Farah and Guzm\'{a}n(2020)]{FG20}
L.~G. Farah and C.~M. Guzm\'{a}n.
\newblock Scattering for the radial focusing inhomogeneous {NLS} equation in
  higher dimensions.
\newblock \emph{Bull. Braz. Math. Soc. (N.S.)}, 51\penalty0 (2):\penalty0
  449--512, 2020.
\newblock ISSN 1678-7544.
\newblock \doi{10.1007/s00574-019-00160-1}.
\newblock URL \url{https://doi.org/10.1007/s00574-019-00160-1}.

\bibitem[Genoud(2012)]{G12}
F.~Genoud.
\newblock An inhomogeneous, {$L^2$}-critical, nonlinear {S}chr\"{o}dinger
  equation.
\newblock \emph{Z. Anal. Anwend.}, 31\penalty0 (3):\penalty0 283--290, 2012.
\newblock ISSN 0232-2064.
\newblock \doi{10.4171/ZAA/1460}.
\newblock URL \url{https://doi.org/10.4171/ZAA/1460}.

\bibitem[Genoud and Stuart(2008)]{GENSTU}
F.~Genoud and C.~A. Stuart.
\newblock Schr\"{o}dinger equations with a spatially decaying nonlinearity:
  existence and stability of standing waves.
\newblock \emph{Discrete Contin. Dyn. Syst.}, 21\penalty0 (1):\penalty0
  137--186, 2008.
\newblock ISSN 1078-0947.
\newblock \doi{10.3934/dcds.2008.21.137}.
\newblock URL \url{https://doi.org/10.3934/dcds.2008.21.137}.

\bibitem[Guzm\'{a}n(2017)]{CARLOS}
C.~M. Guzm\'{a}n.
\newblock On well posedness for the inhomogeneous nonlinear {S}chr\"{o}dinger
  equation.
\newblock \emph{Nonlinear Anal. Real World Appl.}, 37:\penalty0 249--286, 2017.
\newblock ISSN 1468-1218.
\newblock \doi{10.1016/j.nonrwa.2017.02.018}.
\newblock URL \url{https://doi.org/10.1016/j.nonrwa.2017.02.018}.

\bibitem[Guzm\'{a}n and Murphy(2021)]{GM21}
C.~M. Guzm\'{a}n and J.~Murphy.
\newblock Scattering for the non-radial energy-critical inhomogeneous {NLS}.
\newblock \emph{J. Differential Equations}, 295:\penalty0 187--210, 2021.
\newblock ISSN 0022-0396.
\newblock \doi{10.1016/j.jde.2021.05.055}.
\newblock URL \url{https://doi.org/10.1016/j.jde.2021.05.055}.

\bibitem[Murphy(2021)]{M21}
J.~Murphy.
\newblock A simple proof of scattering for the intercritical inhomogeneous
  {NLS}.
\newblock \emph{ArXiv preprint arXiv:2101.04811}, 2021.

\bibitem[Ogawa and Tsutsumi(1991{\natexlab{a}})]{OT91PAMS}
T.~Ogawa and Y.~Tsutsumi.
\newblock Blow-up of {$H^1$} solutions for the one-dimensional nonlinear
  {S}chr\"{o}dinger equation with critical power nonlinearity.
\newblock \emph{Proc. Amer. Math. Soc.}, 111\penalty0 (2):\penalty0 487--496,
  1991{\natexlab{a}}.
\newblock ISSN 0002-9939.
\newblock \doi{10.2307/2048340}.
\newblock URL \url{https://doi.org/10.2307/2048340}.

\bibitem[Ogawa and Tsutsumi(1991{\natexlab{b}})]{ogawa1991blow}
T.~Ogawa and Y.~Tsutsumi.
\newblock Blow-up of {$H^1$} solution for the nonlinear {S}chr\"{o}dinger
  equation.
\newblock \emph{J. Differential Equations}, 92\penalty0 (2):\penalty0 317--330,
  1991{\natexlab{b}}.
\newblock ISSN 0022-0396.
\newblock \doi{10.1016/0022-0396(91)90052-B}.
\newblock URL \url{https://doi.org/10.1016/0022-0396(91)90052-B}.

\bibitem[Strauss(1977)]{Strauss}
W.~A. Strauss.
\newblock Existence of solitary waves in higher dimensions.
\newblock \emph{Comm. Math. Phys.}, 55\penalty0 (2):\penalty0 149--162, 1977.
\newblock ISSN 0010-3616.
\newblock URL \url{http://projecteuclid.org/euclid.cmp/1103900983}.

\bibitem[Weinstein(1982/83)]{W83}
M.~I. Weinstein.
\newblock Nonlinear {S}chr\"{o}dinger equations and sharp interpolation
  estimates.
\newblock \emph{Comm. Math. Phys.}, 87\penalty0 (4):\penalty0 567--576,
  1982/83.
\newblock ISSN 0010-3616.
\newblock URL \url{http://projecteuclid.org/euclid.cmp/1103922134}.

\end{thebibliography}
%\bibliographystyle{abbrvnat}

\newcommand{\Addresses}{{% additional braces for segregating \footnotesize
		\bigskip
		\footnotesize
		
		MYKAEL A. CARDOSO, \textsc{Department of Mathematics, UFPI, Brazil}\par\nopagebreak
		\textit{E-mail address:} \texttt{mykael@ufpi.edu.br}
		
		\medskip
		
		LUIZ G. FARAH, \textsc{Department of Mathematics, UFMG, Brazil}\par\nopagebreak
		\textit{E-mail address:} \texttt{farah@mat.ufmg.br}

}}
\setlength{\parskip}{0pt}
\Addresses

\end{document}